\documentclass{amsart}

\usepackage{amssymb,latexsym}

\newtheorem{lemma}{Lemma}[section]
\newtheorem{theorem}[lemma]{Theorem}
\newtheorem{corollary}[lemma]{Corollary}
\newtheorem{proposition}[lemma]{Proposition}
\newtheorem{definition}[lemma]{Definition}

\newtheorem*{corollaryintroduction}{Corollary 3.5}
\theoremstyle{remark}
\newtheorem*{acknowledgements}{Acknowledgements}
\newtheorem{remark}[lemma]{Remark}
\newtheorem{example}[lemma]{Example}
\DeclareMathOperator{\spec}{\rm Spec}
\DeclareMathOperator{\defo}{\rm Def}
\numberwithin{equation}{section}

\begin{document}
\title[Stable Degenerations of Surfaces Isogenous to a Product II]{Stable Degenerations of Surfaces\\ Isogenous to a Product II}
\address{Liu Wenfei\\School of Mathematics Sciences\\Beijing University\\Beijing 100871\\P.R.China}
\email{liu.wenfei9@gmail.com}
\author{Liu Wenfei}

\begin{abstract}
In this note, we describe the possible singularities on a stable surface which is in the boundary of the moduli space of surfaces isogenous to a product. Then we use the $\mathbb Q$-Gorenstein deformation theory to get some connected components of the moduli space of stable surfaces.
\end{abstract}

\subjclass[2000]{14J10; 14B07}

\maketitle

\section*{Introduction}
Stable surfaces were first introduced by Koll\'ar and Shepherd-Barron in the paper \cite{KoSB}. These surfaces appear naturally as one-parameter limits of surfaces of general type: if $\tilde{\mathcal X}\rightarrow \Delta$ is a one-parameter $\mathbb Q$-Gorenstein family whose general fibres $\tilde{\mathcal X}_t$ are the canonical models of surfaces of general type and $\mathcal X\rightarrow \Delta$ is its canonical model, then the central fibre $\mathcal X_0$ is a stable surface. The singularities on stable surfaces are semi log canonical (cf. Definition \ref{definitionsemilogcanonical}) and a complete classification of them has already been given in \cite{KoSB}. The moduli problem for stable surfaces was solved later through the work of several authors (cf. \cite{Ko90,A94,Hac04,HK04,Ko08,AbH09} and Theorem \ref{compactmoduli}).

In view of the above general framework, examples are badly needed to illustrate the geometry of the moduli space of stable surfaces. One strategy is to compactify the well studied moduli spaces of smooth surfaces whose stable degenerations  can be described in a satisfactory way.  In this direction, van Opstall \cite{vO05,vO06} studied the moduli space of products of curves and the stable degenerations of surfaces isogenous to a product. Recently, Alexeev and Pardini \cite{AP09} gave an explicit compactification of moduli spaces of the more complicated Campedelli and Burniat surfaces, and Rollenske \cite{R09} did the same for very simple Galois double Kodaira fibrations.

Here we continue the work of van Opstall (\cite{vO06}) on the stable degenerations of surfaces isogenous to a product and give some connected components of the moduli space of stable surfaces. A surface isogenous to a product is the quotient of a product of two smooth curves by a free group action. Van Opstall showed that a stable degeneration of such surfaces is the quotient of a product of two stable curves $C\times D$ by a not necessarily free group action (cf. Theorem \ref{vO}). We go one step further to describe explicitly how the group acts on the product of two stable curves. To do this, we first establish some results on the smoothing of a stable curve with a group action (cf. Section \ref{sectionsmoothingcurve}). Then we can apply these results to the free smoothing of a product of two stable curves with a group action (see Sections \ref{sectionsmoothingu} and \ref{sectionsmoothingm}). Note that $X=(C\times D)/G$ is a stable degeneration if and only if $(C\times D,G)$ admits a free smoothing (cf. the discussion after Theorem \ref{vO}). Therefore we get our desired description of the group action on a product of stable curves (Prop. \ref{smoothingu} and \ref{smoothingm}). The basic tool used is Cartan's lemma which reduces the action of a stabilizer to a linear one.

Having described the singularities on a stable degeneration $X$ (Cor. \ref{singularityu} and \ref{singularitym}), we show that if the singularities are of certain type, namely $(U_{1a}),(U_{1b}),(U_{2a})$ or $(U_{2b})$, then the $\mathbb Q$-Gorenstein deformation of $X$ is unobstructed (Theorem \ref{surjective}) and hence the compactification of the moduli space considered by van Opstall  in fact yields connected components of the moduli space of stable surfaces:
\begin{corollaryintroduction}
  Let $S=(C\times D)/G$ be a surface isogenous to a product of unmixed type. Assume the pair $(C,G)$ is a triangle curve (i.e., $C/G\cong\mathbb{P}^1$, and $C\rightarrow C/G$ is branched over $3$ points). Let $M_S^{\text{top}}$ be the moduli space of smooth surfaces with the same topological type as $S$ and $\overline{M_S^{\text{top}}}$ the closure of $M_S^{\text{top}}$ in the moduli space $M^{st}_{a,b}$ with $a=K_S^2,b=\chi(\mathcal{O}_S)$. Then $\overline{M^{\text{top}}_S}$ consists of connected components of  $M^{st}_{a,b}$.
\end{corollaryintroduction}
We should mention that the $\mathbb Q$-Gorenstein deformation theory set out by Hacking (\cite{Hac01} and \cite{Hac04}) is indispensable for our purpose.

\begin{acknowledgements}
 I would like to thank my advisor Prof. Fabrizio Catanese at Universit\"at Bayreuth for suggesting this research and for his patience and encouragement during many discussions. The idea of classifying the singularities on stable degenerations of surfaces isogenous to a product originates from him. I want to thank Fabio Perroni and S\"{o}nke Rollenske for several discussions during the preparation of this note. I am grateful to Stephen Coughlan for improving the English presentation. Thanks also go to my domestic advisor Prof. Jinxing Cai at Beijing University for encouragement.

This work was completed at Universit\"{at} Bayreuth under the financial support of China Scholarship Council \textquotedblleft High-level university graduate program\textquotedblright and DFG Forschergruppe 790 \textquotedblleft Classification of algebraic surfaces and compact complex manifolds\textquotedblright.
\end{acknowledgements}

\section{Preliminaries}
\subsection{Notation}
Let $G$ be a finite group acting on a set $A$.

For $\sigma \in G$, $|\sigma|$ is the order of $\sigma$.

For $a\in A$, $G_a:=\{g\in G|g\cdot a =a\}$ is the stabilizer of $a$. If $G_a\neq \{1\}$, we say that $a$ is a fixed point of the action. If $G_a=\{1\}$ for every $a\in A$, we say that $G$ acts freely on $A$.

$\mathbb Z_n$ denotes the cyclic group of order $n$.

We work over the field $\mathbb C$ of complex numbers.

\subsection{Surfaces isogenous to a product}\label{isogenoussurface}
\begin{definition}[\cite{Cat00}, Definition 3.1]\label{definitionisogenous}
A smooth projective surface $S$ is isogenous to a (higher) product if it is a quotient $S=(C\times D)/G$, where $C,D$ are smooth curves of genus at least two, and $G$ is a finite group acting freely on $C\times D$.
\end{definition}

Let $S=(C\times D)/G$ be a surface isogenous to a product. Let $G^\circ:=G\cap (\text{Aut} (C)\times \text{Aut}(D))$; then $G^\circ$ acts on the two factors $C,D$ and  acts on $C\times D$ via the diagonal action.  We say that $S$ is of \emph{unmixed} type if $G=G^\circ$; otherwise $S$ is said to be of \emph{mixed} type. By \cite[Prop. 3.13]{Cat00}, we always assume $G^\circ$ acts faithfully on  $C,D$.

By  \cite[Prop. 3.16]{Cat00}, surfaces $S=(C\times D)/G$ isogenous to a product and of mixed type are obtained as follows. There is a (faithful) action of a finite group $G^\circ$ on a curve $C$ of genus at least 2 and a nonsplit extension
\[
 1\rightarrow G^\circ\rightarrow G\rightarrow \mathbb{Z}_2 \rightarrow 1,
\]
yielding a class $[\varphi]$ in $\text{Out}(G^\circ)=\text{Aut}(G^\circ)/\text{Int}(G^\circ)$, which is of order $\leq 2.$ Once we fix a representative $\varphi$ of the above class, there exists an element $\tau'$ in $G\setminus G^\circ$ such that, setting $\tau=\tau'^2$, we have:
\begin{enumerate}
 \item [(I)]{$\varphi(\gamma)=\tau'\gamma\tau'^{-1}$},
 \item [(II)]{$G$ acts, under a suitable isomorphism of $C$ and $D$, by the formulae: $\gamma(P,Q)=(\gamma P,(\varphi\gamma) Q)$ for $\gamma$ in $G^\circ$; whereas the lateral class of $G^\circ$ consists of the transformations
\[
 \tau'\gamma(P,Q)=((\varphi\gamma)Q,\tau\gamma P).
\]}
\end{enumerate}
 Let $\Gamma$ be the subset of $G^\circ$ consisting of the transformations having some fixed point. Then the condition that $G$ acts freely amounts to:
\begin{enumerate}
 \item [(A)]{$\Gamma\cap\varphi(\Gamma)=\{1\}$}.
 \item [(B)]{there is no $\gamma$ in $G^\circ$ such that $\varphi(\gamma)\tau\gamma$ is in $\Gamma$.}
\end{enumerate}

The moduli space of surfaces isogenous to a product is partly illustrated in the following abridged theorem:
\begin{theorem}[\cite{Cat03}]\label{moduliisogenous}
 Let $S$ be a surface isogenous to a product. Then moduli space $M^{\text{top}}_S$ of surfaces with the same topological type as $S$ is either irreducible and connected or it contains two connected components which are interchanged by complex conjugation.
\end{theorem}

\subsection{Cartan's lemma}
The following lemma is used throughout Section 2 for the (analytically) local analysis of the group actions.
\begin{lemma}[Cartan's lemma]
Let $z\in Z$ be an analytic singularity with Zariski tangent space $T$ and let $G$ be a finite group of automorphisms of $z\in Z$. Then there exists a $G$-equivariant embedding $z\in Z\rightarrow T \ni 0$.
\end{lemma}

\subsection{$\mathbb{Q}$-Gorenstein deformation theory of stable surfaces}\label{sectionqgorensteindeformation}
 We recall the $\mathbb{Q}$-Gorenstein deformation theory of stable surfaces set out by Hacking.

\begin{definition}\label{definitionsemilogcanonical}
A (reduced) surface $X$ is said to have semi log canonical (slc) singularities if
\begin{enumerate}
 \item [(i)]{$X$ is $S_2$;}
 \item [(ii)]{$X$ has at most double normal crossing singularities in codimension $1$;}
 \item [(iii)]{the Weil divisor class $K_X$ is $\mathbb{Q}$-Cartier;}
 \item [(iv)]{if $\tilde{X}\rightarrow X$ is the normalization and $\tilde{D}\subset \tilde{X}$ is the preimage of the codimension-1 part of $X_{sing}$, then the pair $(\tilde{X},\tilde{D})$ is log canonical, i.e., for any resolution $\mu\colon\hat{X}\rightarrow \tilde{X}$, we have
\[
 K_{\hat{X}}+\mu^{-1}_*\tilde{D}=\mu^*(K_{\tilde{X}}+\tilde{D})+\sum a_iE_i
\]
with all $a_i\geq -1$.}
\end{enumerate}
 A stable surface is a projective slc surface with an ample dualizing sheaf.
\end{definition}

 Let $\mathcal{F}$ be a coherent sheaf. For $n\in\mathbb Z$, we define the $n$-th reflexive power of $\mathcal{F}$ by $\mathcal{F}^{[n]}:=(\mathcal{F}^{\otimes n})^{**}$, the double dual of the $n$-th tensor product. Let $\mathcal R$ be the category of Noetherian local $\mathbb C$-algebras with residue field $\mathbb C$.

\begin{definition}\label{definitionqgorensteinfamily}
Let $B$ a Noetherian scheme over $\mathbb C$. A flat projective morphism $f\colon\mathcal X\rightarrow B$ is called a $\mathbb Q$-Gorenstein family of stable surfaces if for every closed point $t\in B$, the fibre $\mathcal X_t$ is a stable surface and for every $n\in\mathbb Z$,   $\omega^{[n]}_{\mathcal X/B}$ commutes with any base change.

Let $R\in\mathcal R$.  A $\mathbb Q$-Gorenstein deformation of $X$ over $R$ is a $\mathbb Q$-Gorenstein family $f\colon \mathcal X\rightarrow \spec R$ together with an isomorphism $\mathcal X\otimes_R k(R)\cong X$. Isomorphisms between two $\mathbb Q$-Gorenstein deformations are defined in an obvious way.
\end{definition}

\begin{remark}
The hypothesis that $\omega^{[n]}_{\mathcal X/B}$ commutes with any base change for any $n$ is  called \emph{Koll\'ar's condition} and it implies that $\omega^{[N]}_{\mathcal X/B}$ is invertible for some $N$.
\end{remark}

For a stable surface $X$, we define a functor $\mathcal Def^{QG}_X\colon \mathcal R \rightarrow (Sets)$ as follows: for any $R\in \mathcal R$,
\[
 \mathcal Def^{QG}_X(R)=\{ \text{$\mathbb Q$-Gorenstein deformations of $X$ over $R$}\}/\simeq.
\]
The $\mathbb Q$-Gorenstein deformations of $X$ are exactly those deformations which, locally at each point $P\in X$, are induced by equivariant deformations of the canonical covering of  $P\in X$ (\cite[Prop. 10.13]{Hac01}). If  $X$ is Gorenstein, then the $\mathbb Q$-Gorenstein deformation theory of $X$ coincides with the ordinary one.
\begin{remark}
If a \emph{semiuniversal} $\mathbb Q$-Gorenstein deformation exists, we denote the base by $\defo^{QG}_X$. When we refer to the ordinary deformation of $X$, the semiuniversal base is denoted by $\defo_X$.
\end{remark}

Set $T^1_{QG,X}=\mathcal Def^{QG}_X(\mathbb C[\epsilon]/\epsilon^2)$, the space of first-order $\mathbb Q$-Gorenstein deformations. Let $\mathcal T^1_{QG,X}$ be the sheaf associated to the presheaf defined by
\[
 U\rightarrow \mathcal Def^{QG}_U(\mathbb C[\epsilon]/\epsilon^2),
\]
for any open subset $U$ of  $X$.  Then we have a local-to-global sequence:
\[
 0\rightarrow H^1(\mathcal{T}_X)\rightarrow T^1_{QG,X}\rightarrow H^0(\mathcal{T}^1_{QG,X})\rightarrow H^2(\mathcal{T}_X),
\]
where $\mathcal T_X=\mathcal Hom_{\mathcal O_X}(\Omega_X,\mathcal O_X)$ is the tangent sheaf of $X$  (\cite[Page 227]{Hac04}). Note that $T^1_{QG,X}$ is considered as the tangent space to the functor $\mathcal Def^{QG}_{X}$.

\subsection{Stable surfaces in this paper}\label{sectionstablesurface}
We are mostly interested in stable surfaces that are quotients of a product of two stable curves. We recall the definition of a stable curve first (\cite{DM}):
\begin{definition}
Let $g\geq 2$ be an integer. A stable curve of genus $g$ is a reduced, connected, 1-dimensional scheme $C$ over $\mathbb{C}$ such that:
\begin{enumerate}
 \item [(i)]{$C$ has only ordinary double points as singularities;}
 \item [(ii)]{if $E$ is a non-singular rational component of $C$, then $E$ meets the other components of $C$ in more than 2 points;}
 \item [(iii)]{$\text{dim H}^1(\mathcal{O}_C)=g$.}
\end{enumerate}
\end{definition}

\begin{proposition}\label{productofcurves}
 Let $C,D$ be stable curves. Then $C\times D$ is a stable surface.
\end{proposition}
\begin{proof}
See \cite[Prop. 3.1]{vO05}.
\end{proof}

\begin{proposition}
 Let $Z$ be a stable surface and $G$ a finite group acting on $Z$ with finitely many fixed points. Then $Z/G$ is also a stable surface.
\end{proposition}
\begin{proof}
See the proof of \cite[Theorem 3.1]{vO06}.
\end{proof}

\begin{corollary}
Let $C,D$ be stable curves. Let $Z:=C\times D$ and $G$ a group acting on $Z$ with finitely many fixed points. Then $(C\times D)/G$ is a stable surface.
\end{corollary}

The stable degenerations of surfaces isogenous to a product are of the above form, as we will see from the following:
\begin{theorem}[\cite{vO06}, Theorem 3.1]\label{vO}
 Suppose $\mathcal X^*\rightarrow \Delta^*$ is a family of surfaces isogenous to a product over a  punctured disk. Then, possibly after a finite change of base, totally ramified over the origin, $\mathcal{X}$ can be completed to a family of stable surfaces over the disk whose central fibre is a quotient of a product of stable curves.
\end{theorem}
According to the proof of the above theorem in \cite{vO06}, we give an explicit construction of the stable degenerations here. There are two cases:

\textbf{Unmixed case}: we have, up to finite base change, $G$-equivariant smoothings of stable curves (cf. Section 2.1) $\mathcal{C}\rightarrow\Delta$ and  $\mathcal{D}\rightarrow\Delta$ such that the completion $\mathcal X\rightarrow \Delta$ of $\mathcal X^*\rightarrow\Delta^*$ is  $(\mathcal{C}\times_{\Delta}\mathcal{D})/G\rightarrow \Delta.$ In particular, the central fibre is $(\mathcal C_0\times \mathcal D_0)/G$ where $G$ acts faithfully on $\mathcal C_0, \mathcal D_0$ and acts diagonally on $\mathcal C_0\times \mathcal D_0$. However, the action of $G$ on $\mathcal C_0\times \mathcal D_0$ is not necessarily free.

\textbf{Mixed case}: there exists a finite group $G^{\circ}$, a $G^{\circ}$-equivariant smoothing $\mathcal{C}\rightarrow\Delta$ of stable curves and a nonsplit extension
\[
 1\rightarrow G^{\circ}\rightarrow G\rightarrow\mathbb{Z}_2\rightarrow 1
\]
yielding an automorphism $\varphi$ of $G^{\circ}$, such that the pairs $(\mathcal{C}_t,G)$ with $t\neq 0$ satisfy properties $(I),(II),(A),(B)$ in Section \ref{isogenoussurface}. On the central fibre $\mathcal{C}_0$, we still have a $G$-action that enjoys properties $(I),(II)$, but not necessarily $(A),(B)$, i.e., the action of $G$ on $\mathcal{C}_0\times\mathcal{C}_0$ is not necessarily free. Now the completion $\mathcal X\rightarrow\Delta$ is $(\mathcal{C}\times_{\Delta}\mathcal{C})/G\rightarrow\Delta$, where the action of $G^\circ$ on the second factor is twisted by $\varphi$.

In both cases, the stable degeneration $\mathcal{X}_0$ is of the form $(C\times D)/G$, where $C,D$ are stable curves and $G$ acts on $C\times D$ with finitely many fixed points. Tautologically the pair $(C\times D,G)$ admits a free smoothing, i.e., a one-parameter family $\mathcal{C}\times_{\Delta}\mathcal{D}\rightarrow \Delta$ such that the following hold:
\begin{enumerate}
  \item[(i)]{$\mathcal{C}_0\times\mathcal{D}_0\cong C\times D$;}
 \item[(ii)]{The fibre $\mathcal{C}_t\times\mathcal{D}_t$ over $t\neq 0$ is smooth;}
 \item[(iii)]{$G$ acts on $\mathcal{C}\times_\Delta\mathcal{D}$ preserving the fibres and the action of $G$ on the central fibre coincides with the given action of $G$ on $C\times D$;}
 \item[(iv)]{$G$ acts freely on the general fibres $\mathcal{C}_t\times\mathcal{D}_t$ for $t\neq 0$.}
\end{enumerate}

\subsection{The moduli space of stable surfaces}\label{sectionmoduli}
Let $(Sch)/\mathbb{C}$ be the category  of Noetherian $\mathbb{C}$-schemes. We define the moduli functor of stable surfaces: for any $B\in(Sch)/\mathbb{C}$,
\begin{align*}
 \mathcal{M}_{a,b}^{st}(B)=\{\mathcal{X}/B\ |\ \mathcal{X}/B \text{ is a $\mathbb{Q}$-Gorenstein family of stable surfaces over } B \\
\text{ and for any closed point } t\in B,K_{\mathcal{X}_t}^2=a,\chi(\mathcal{O}_{\mathcal{X}_t})=b   \}/\cong.
\end{align*}

\begin{theorem}\label{compactmoduli}
 There is a projective coarse moduli space $M^{st}_{a,b}$ for $\mathcal{M}_{a,b}^{st}$.
\end{theorem}
\begin{proof}
 See \cite[Section 7.C]{Kov09} and \cite{Ko08}.
\end{proof}

We shall get some connected components of this moduli space by studying the $\mathbb{Q}$-Gorenstein deformations of stable degenerations of surfaces isogenous to a product in Section 3.

\section{Stable degenerations of surfaces isogenous to a product}
In this section, we will give a precise description of possible singularities on stable degenerations of surfaces isogenous to a product. This is a careful improvement of Theorem \ref{vO}, which allows one to further study the $\mathbb Q$-Gorenstein deformations of the stable degenerations.

\subsection{Smoothings of stable curves with group actions}\label{sectionsmoothingcurve}

Our surfaces can be constructed by taking finite quotients of products of two stable curves, so their geometry is closely related to that of stable curves. In this section, we will establish some facts about smoothings of stable curves with group actions. More precisely, in the case when the group action admits a smoothing, we will show what the stabilizers on the central fibre can be and how they act locally analytically. These facts are used in Sections 2.2 and 2.3 for the smoothing of a product of stable curves with a group action.

\begin{definition}
 Let $G$ be a finite group acting faithfully on a stable curve $C$. A smoothing of the pair $(C,G)$ is a (flat) family of stable curves $\mathcal{C}\rightarrow \Delta$ over the unit disk such that
\begin{enumerate}
 \item[(i)]{the central fibre $\mathcal{C}_0$ is isomorphic to $C$;}
 \item[(ii)]{The fibre $\mathcal{C}_t$ of the family over $t\neq 0$ is smooth;}
 \item[(iii)]{$G$ acts on $\mathcal{C}$ preserving the fibres, and the action on $\mathcal{C}_0$ coincides with the given one on $C$ under the isomorphism of (i).}
\end{enumerate}
\end{definition}

\begin{remark}
 We also call $\mathcal{C}\rightarrow \Delta$ a $G$-equivariant smoothing of $C$.
\end{remark}

Now let $(C,G)$ be as in the definition and assume $\mathcal{C}\rightarrow \Delta$ is a smoothing of $(C,G)$.
\begin{lemma}\label{lemma1}
 There are only finitely many points on $C$ having non-trivial stabilizers, or, equivalently, there are only finitely many fixed points for the $G$-action.
\end{lemma}
\begin{proof}
 Otherwise there is a $\tau\neq 1\in G$ acting as identity on some irreducible component $D$ of $C$. Pick a smooth point $P$ of $D$. Then $\mathcal{C}$ is smooth around $P$ and we can take local coordinate $z$ of $D$ and local coordinate $t$ of $\Delta$ such that $(z,t)$ form local coordinates of $\mathcal{C}$ around $P$ and  $\tau$ acts as $(z,t)\mapsto (z,t)$. So $\tau=1\in G$ is the identity, a contradiction.
\end{proof}

\begin{lemma}\label{lemma2}
If $P\in C$ is a smooth point, then $G_P$ is cyclic.
\end{lemma}
\begin{proof}
 There is an embedding $G_P\hookrightarrow GL(T_P C)$, where $T_P C\cong \mathbb C$ is the tangent space of $C$ at $P$. Since $T_P C$ is a 1-dimensional vector space, $GL(T_PC)\cong\mathbb{C}^*$. So $G_P$, being a finite subgroup of $\mathbb{C}^*$, is cyclic.
\end{proof}

\begin{lemma}\label{lemma3}
If $P\in C$ is a smooth point, and $1\neq\tau\in G_P$, then
$\tau$ also fixes points on $\mathcal{C}_t$, for $t\neq 0$.
\end{lemma}
\begin{proof}
 As in the proof of Lemma \ref{lemma1}, we can find local coordinates $(z,t)$ for $\mathcal{C}$ around $P$ such that $\tau$ acts as $(z,t)\mapsto(\xi z,t)$, where $\xi\in\mathbb{C}^*$ is a primitive $|\tau|$-th root of unity. So $\tau$ fixes $(0,t)\in\mathcal C_t$, for $t\neq 0$.
\end{proof}

\begin{lemma}\label{lemma4}
 If $P\in C$ is a node, then $G_P$ is either cyclic or dihedral.
\end{lemma}
\begin{proof}
The germ of $\mathcal{C}$ around $P$ can be seen as a deformation of a node. We can find an embedding of the germ into $(\mathbb{C}^3,0)$ such that the equation of the germ is $xy-t^k=0,k\geq 1$. In fact, let
\[\begin{matrix}
\{xy-s=0\} & \rightarrow & \Delta\\
(x,y,s)   & \mapsto & s
\end{matrix}\]
be the semiuniversal family of a node. Then locally around $P$, $\mathcal{C}\rightarrow \Delta$ is just the pull-back by
\[
\begin{matrix}
 \Delta & \rightarrow & \Delta \\
  t     & \mapsto & s=t^k,
 \end{matrix}
\]
where $k\geq 1.$

By Cartan's lemma, we can assume the action of $G_P$ is given by
\[
\tau\colon (x,y,t)\mapsto (a_1 x+a_2 y, b_1 x+b_2 y, t).
\]
for any $\tau\in G_P$. Since $G_P$ acts on the central fibre $\mathcal{C}_0:xy=0$, we have
\[\left(
 \begin{matrix}
  a_1 & a_2 \\
  b_1 & b_2
 \end{matrix}\right)
 =\left(\begin{matrix}
   \xi & 0   \\
   0   & \xi^{-1}
  \end{matrix}\right)
\text{or} \left(\begin{matrix}
   0 &  \eta \\
   \eta^{-1} & 0
  \end{matrix} \right),
\]
where $\xi$ is a  primitive $|\tau|$-th root of unity and $\eta$ is some non-zero number. Let $\pi:G_P \rightarrow \mathbb{Z}_2$ be the determinant homomorphism: for any $\tau\in G_P$,
\[
\pi(\tau):=det(\tau)=
 \begin{cases}
  1,  & \text{if } \tau \text{ does not interchange the branches at } P,\\
  -1, & \text{if } \tau \text{ interchanges the branches at }P.
 \end{cases}
\]
Then the kernel $H $ of $\pi$ consists of $\tau\in G_P$ whose action is given by
$\left(\begin{matrix}
   \xi & 0   \\
   0   & \xi^{-1}
  \end{matrix}\right).$
Note that $H$ embeds into $\mathbb{C}^*$:
\[
 \begin{matrix}
  H &\rightarrow &\mathbb{C}^*\\
  \tau &\mapsto &\xi.
 \end{matrix}
\]
If $\text{im}(\pi)=\{1\}$, then $G_P=H$ is cyclic. If $\text{im}\pi=\mathbb{Z}_2$, then there exists $\tau\in G_P$ such that
\[
 \tau(x,y,t)=(\eta y,\eta^{-1} x, t),
\]
 and it is easy to see in this case that
\[
G_P=
 \begin{cases}
  \text{a dihedral group, }& \text{if  $|H|\geq 2$};\\
  \mathbb{Z}_2, &  \text{if  $|H|=1$}.
 \end{cases}
\]
\end{proof}

\begin{lemma}\label{lemma5}
 Let $P\in C$ be a node. Suppose $1\neq\tau\in G_P$: then $\tau$ fixes points on $\mathcal{C}_t$  for $t\neq 0$ if and only if $\pi(\tau)=-1$, where $\pi\colon G_P\rightarrow \mathbb{Z}_2$ is as in the proof of the previous lemma.
\end{lemma}
\begin{proof}
We adopt the notation in  Lemma \ref{lemma4}, so the germ of $\mathcal{C}$ around $P$ is defined by $xy-t^k=0,k\geq 1$ and the action of $G_P$ is linear.

Let $1\neq\tau\in G_P$. If $\pi(\tau)=1$, then $\tau(x,y,t)=(\xi x,\xi^{-1}y,t)$,
where $\xi\in\mathbb{C}^*$ is a  primitive $|\tau|$-th root of unity. Suppose $(x,y,t)\in\mathcal{C}$ is a fixed point of $\tau$. Then
\[
 \tau(x,y,t)=(\xi x,\xi^{-1}y,t)=(x,y,t)\Rightarrow x=y=0,
\]
and $xy=t^k$ implies $t=0$. So $\tau$ does not fix any point on $\mathcal{C}_t$ for $t\neq 0$.

On the other hand, if $\pi(\tau)=-1$, then $\tau(x,y,t)=(\eta y,\eta^{-1}x,t)$ with $\eta \in\mathbb{C}^*$. So $\tau(x,y,t)=(x,y,t)$  if and only if $\eta y=x$. Taking the equation $xy=t^k$ into consideration, $\tau$ fixes 2 points: $(\eta\sqrt{\frac{t^k}{\eta}},t)$ and $(-\eta\sqrt{\frac{t^k}{\eta}},t)$ on $\mathcal{C}_t$, for $t\neq0$.
\end{proof}

Now we can state our main theorem in this section:
\begin{theorem}\label{smoothingcurve}
A pair (C,G) admits a smoothing if and only if for any node $P\in C$, we can find local (analytic)  embedding of $C\colon (xy=0)\subset\mathbb{C}^2$ such that, for any $\tau\in G_P$, the action of $\tau$ is given by either
\begin{enumerate}
 \item [(i)] {$(x,y)\mapsto(\xi x,\xi^{-1}y)$} where $\xi$ is a $|\tau|$-th root of unity; or
 \item [(ii)]{$(x,y)\mapsto(\eta y,\eta^{-1}x)$} where $\eta\in\mathbb{C}^*$ is a nonzero number.
\end{enumerate}
\end{theorem}
\begin{proof}
 The \textquotedblleft only if\textquotedblright part is shown in the proof of Lemma \ref{lemma4}.

For the \textquotedblleft if \textquotedblright part, we divide the proof into two steps.

\textbf{Step 1:} The germ $P\in C$ has a local $G$-equivariant smoothing. More precisely, let $U\subset C$ be a neighborhood around $P$ defined by $xy=0\subset\mathbb{C}^2$ as in the hypothesis, we will show that the pair $(U,G_P)$ is $G_P$-smoothable. In fact we can consider the family $\mathcal{U}:(xy-s=0)\subset \mathbb{C}^2\times\Delta\rightarrow\Delta$ with $s\in\Delta$ as the parameter. For any $\tau\in G_P$,
\[
 \tau(x,y)=(\xi x,\xi^{-1}y) \text{ or } (\eta y,\eta^{-1}x),
\]
 and it is easily seen that the action of $G_P$ on $U$ extends to the family $\mathcal{U}\rightarrow \Delta$.

Note that $\mathcal{U}\rightarrow \Delta$ is the semiuniversal deformation of the node $P\in U$ and the tangent space of the base space at 0 is $\mathcal{E}xt^1_{\mathcal{O}_U}(\Omega_U,\mathcal{O}_U)\cong\mathbb{C}$.
The fact that $(U,G_P)$ is smoothable means exactly that $\mathcal{E}xt^1_{\mathcal{O}_U}(\Omega_U,\mathcal{O}_U)$  is $G_P$-invariant.

\textbf{Step 2:} We will use the local-to-global exact sequence
\[
  0 \rightarrow H^1(C,\mathcal T_C)\rightarrow \text{Ext}^1_{\mathcal{O}_C}(\Omega_C,\mathcal{O}_C) \xrightarrow{\pi} H^0(C,\mathcal{E}xt^1_{\mathcal{O}_C}(\Omega_C,\mathcal{O}_C))\rightarrow 0
\]
to prove that local smoothings of nodes with stabilizers lift to  a smoothing of $(C,G)$.
To do this, first note that
\begin{equation}\label{directsum}
 H^0(C,\mathcal{E}xt^1_{\mathcal{O}_C}(\Omega_C,\mathcal{O}_C))=\bigoplus_{P\text{ node}}\text{Ext}^1_{\mathcal{O}_{C,P}}(\Omega_{C,P},\mathcal{O}_{C,P}),
\end{equation}
where, for any coherent sheaf $\mathcal{F}$ on $C$, $\mathcal{F}_P$ denotes the stalk of $\mathcal{F}$ at $P$. For any $\tau\in G$, $\tau$ acts on $H^0(C,\mathcal{E}xt^1_{\mathcal{O}_C}(\Omega_C,\mathcal{O}_C))$ and maps
the $\text{Ext}^1_{\mathcal{O}_{C,P}}(\Omega_{C,P},\mathcal{O}_{C,P})$ summand isomorphically to the $\text{Ext}^1_{\mathcal{O}_{C,\tau(P)}}(\Omega_{C,\tau(P)},\mathcal{O}_{C,\tau(P)})$ summand.

Let $n(P):=|G/G_P|$ and $\tau_1,\dots,\tau_{n(P)}\in G$  representatives of elements of $G/G_P$. Then $\tau_1(P),\dots,\tau_{n(P)}(P)$ is the orbit of $P$ under the action of $G$. And $G$ acts on the vector space
\[
 V_P:=\bigoplus_{j=1}^{n(P)}\text{Ext}^1_{\mathcal{O}_{C,\tau_j(P)}}(\Omega_{C,\tau_j(P)},\mathcal{O}_{C,\tau_j(P)}).
\]
The invariant subspace $V_P^G$ is 1-dimensional, spanned by
\[
 (\tau_1(\sigma),\dots,\tau_{n(P)}(\sigma)),
\]
where $\sigma$ is an element spanning $\text{Ext}^1_{\mathcal{O}_{C,P}}(\Omega_{C,P},\mathcal{O}_{C,P})\cong\mathbb{C}$. In view of (\ref{directsum}), the dimension of $H^0(C,\mathcal{E}xt^1_{\mathcal{O}_C}(\Omega^1_C,\mathcal{O}_C))$ is exactly the number of node orbits under the action of $G$. Taking the $G$-invariants of the local-to-global sequence, we get
\[
  0 \rightarrow H^1(C,\mathcal T_C)^G \rightarrow \text{Ext}^1_{\mathcal{O}_C}(\Omega_C,\mathcal{O}_C)^G \xrightarrow{\pi} H^0(C,\mathcal{E}xt^1_{\mathcal{O}_C}(\Omega_C,\mathcal{O}_C))^G\rightarrow 0.
\]
In particular, there exists $\lambda\in\text{Ext}^1_{\mathcal{O}_C}(\Omega_C,\mathcal{O}_C)^G$ such that the $\pi(\lambda)$'s $P$-summand is nonzero for any node $P \in C$.
Then $\lambda$ gives a smoothing of $(C,G)$.
\end{proof}

\subsection{Singularities of degenerations of surfaces isogenous to a product of unmixed type}\label{sectionsmoothingu}
We study smoothings of products of two stable curves with a group action in a similar way as $G$-equivariant smoothings of curves in Section \ref{sectionsmoothingcurve}.

We treat the unmixed case first.
\begin{proposition}[Criterion for free smoothings in the unmixed case]\label{smoothingu}
Let $C, D$ be two stable curves and let $G$ be a finite group acting on $C$ and $D$.
Let $G$ act on $C \times D$ diagonally. Then the pair $(C\times D,G)$ admits a free smoothing if and only if for any $(P,Q)\in C\times D$, we have one of the following:
\begin{enumerate}
 \item[$(U_0)$]{if both $P,Q$ are smooth points on $C,D$ respectively, then the stabilizer $G_{(P,Q)}=\{1\}$.}
 \item[$(U_1)$]{if one of $P,Q$, say $P$, is a node and the other is a smooth point, then $G_{(P,Q)}=\langle\tau\rangle$ is cyclic, and we can find a local embedding of $C:(xy=0)\subset \mathbb{C}^2$ as well as a local coordinate $z$ of $D$ such that $\tau (x,y,z)=(\xi  x,\xi^{-1}y,\xi^qz)$, where $\xi$ is a primitive root of unity of order $|\tau|$ and $(q,|\tau|)=1$}.
 \item[$(U_2)$]{if both $P,Q$ are nodes of respective curves, then $G_{(P,Q)}=\langle\tau\rangle$ is cyclic and $\tau$ interchanges the branches of at most one of $C$ and $D$. In this case, we have one of the following}
\begin{enumerate}
\item[$(U_{2a})$]{$G_{(P,Q)}=\{1\}$.}
\item[$(U_{2b})$]{if $\tau$ does interchange the branches of $C$ or $D$, say $C$, then the order of $\tau$ is 2 and we can choose local embeddings $C:(xy=0)\subset \mathbb C^2$ and $D:(zw=0)\subset\mathbb C^2$ such that $\tau(x,y,z,w)=(y,x,-z,-w)$}.
\item[$(U_{2c})$]{if $\tau\neq 1$ and it does not interchange any branches of $C,D$, then we can choose local embeddings $C:(xy=0)\subset \mathbb C^2$ and $D:(zw=0)\subset\mathbb C^2$ such that $\tau(x,y,z,w)=(\xi x,\xi^{-1}y,\xi^qz,\xi^{-q}w)$, where $\xi$ is a primitive root of unity of order $|\tau|$ and $(q,|\tau|)=1$.}
\end{enumerate}
\end{enumerate}
\end{proposition}
\begin{proof}
For the $``\Rightarrow"$ direction, suppose $\mathcal{Z}=\mathcal{C}\times_{\Delta}\mathcal{D}\rightarrow \Delta$ is a free smoothing of $(C\times D, G)$. Let $(P,Q)$ be any point on $C\times D$. Note that $G_{(P,Q)}=G_P\cap G_Q$. We divide our further discussion into 3 cases:

\textbf{Case \text{$(U_0)$}}: $P,Q$ are both smooth points on $C,D$ respectively.

We will show in this case that $G_{(P,Q)}=\{1\}$. Suppose $1\neq \tau\in G_{(P,Q)}$: then $P,Q$ are both fixed points of $\tau$. Since $P,Q$ are both smooth points on $C,D$ respectively, $\tau$ fixes points on $\mathcal{C}_t, \mathcal{D}_t$ for $t\neq 0$ by Lemma \ref{lemma3}. So $\tau$ fixes points on $\mathcal{C}_t\times\mathcal{D}_t$ for $t\neq 0$, which contradicts the assumption that $G$ acts freely on $\mathcal{C}_t\times\mathcal{D}_t$ for $t\neq 0$.

\textbf{Case \text{$(U_1)$}}: One of $P,Q$, say $P$, is a node and the other is a smooth point.

Suppose $G_{(P,Q)}\neq\{1\}$. By Lemma \ref{lemma2}, $G_Q$ is cyclic and hence its subgroup $G_{(P,Q)}$ is also cyclic. Let $G_{(P,Q)}=\langle\tau\rangle,\tau\neq 1.$ By Lemma \ref{lemma3}, $\tau$ fixes points of $\mathcal{D}_t$ for $t\neq 0$. Since $G$ acts freely on $\mathcal{C}_t\times\mathcal{D}_t$ for $t\neq 0$, $\tau$ does not fix any point of $\mathcal{C}_t,t\neq 0.$ By Lemma \ref{lemma5}, $\tau$ does not interchange the branches of $C$ around $P$. Hence there are a local embedding $C:(xy=0)\subset\mathbb{C}^2$ around $P$  and a local coordinate $z$ of $D$ around $Q$ such that the action of $\tau$ on $C\times D$ is
\[
(x,y,z)\mapsto (\xi x,\xi^{-1}y,\xi^q z)
\]
where $\xi$ is a primitive root of unity of order $|\tau|$ and $(q,|\tau|)=1$.

\textbf{Case \text{$(U_2)$}}: $P,Q$ are both nodes on $C,D$.

Suppose $G_{(P,Q)}\neq\{1\}.$ By Lemma \ref{lemma4}, we have that $G_{P}$ and $G_{Q}$ are either cyclic or dihedral. This implies that $G_{(P,Q)}=G_{P}\cap G_{Q}$ is either cyclic or dihedral. Suppose $G_{(P,Q)}$ is dihedral. Then $G_{P}$ and $G_{Q}$ are both dihedral and, by the proof of Lemma \ref{lemma4}, there is $\tau_1\in G_{(P,Q)}$ (resp. $\tau_2\in G_{(P,Q)}$) such that $\tau_1$ (resp. $\tau_2$) interchanges the branches of $C$ at $P$ (resp. the branches of $D$ at $Q$). By Lemma \ref{lemma5}, $\tau_1$ (resp. $\tau_2$) fixes points of $\mathcal{C}_t$ (resp. $\mathcal{D}_t$) for $t\neq 0$. Since neither $\tau_1$ nor $\tau_2$ fixes points on $\mathcal{C}_t\times \mathcal{D}_t$, $\tau_1$ (resp. $\tau_2$) does not fix points on $\mathcal{D}_t$ (resp. $\mathcal{C}_t$). Again by Lemma \ref{lemma5}, $\tau_1$ (resp. $\tau_2$) does not interchange the branches of $D$ at $Q$ (resp. the branches of $C$ at $P$). Now set $\tau:=\tau_1\tau_2$, then $\tau$ interchanges the branches of $C$ as well as those of $D$. This implies that $\tau$ fixes points on $\mathcal{C}_t\times \mathcal{D}_t$ for $t\neq 0$, a contradiction.

So $G_{(P,Q)}$ is cyclic and we can assume that $G_{(P,Q)}=\langle\tau\rangle$. Since $\tau$ does not fix any point on $\mathcal{C}_t\times\mathcal{D}_t$ for $t\neq 0$, $\tau$ interchanges the branches of at most one of $C$ and $D$. If $\tau$ does interchange the branches of one of $C$ and $D$, say $C$, then the order of $\tau$ is 2 (Lemma \ref{lemma4}) and we can choose local embeddings $C\colon (xy=0)\subset \mathbb C^2$ and $D\colon (zw=0)\subset \mathbb C^2$ such that $\tau$ acts as
\[
 (x,y,z,w)\rightarrow(y,x,-z,-w).
\]

If $\tau$ does not interchange any branches of $C,D$, then we can choose local embeddings $C:(xy=0)\subset \mathbb{C}^2,D:(zw=0)\subset \mathbb{C}^2$ such that
\[
\tau(x,y,z,w)=(\xi x,\xi^{-1}y,\xi^q z,\xi^{-q}w),
\]
 where $\xi$ is a primitive root of unity of order $|\tau|$ and $(q,|\tau|)=1.$

For the $``\Leftarrow"$ direction, note that $C$ and $D$ admit $G$-equivariant smoothings $\mathcal{C}\rightarrow\Delta,\mathcal{D}\rightarrow \Delta$ by Theorem \ref{smoothingcurve}. In each of the cases $(U_0)$, $(U_1)$, $(U_2)$, any non-trivial element $\tau^k\in G_{(P,Q)}=\langle\tau\rangle$ interchanges at most the local branches of one of the factors. This guarantees that $\tau^k$ acts locally freely on at least one of the factors of $\mathcal{C}_t\times\mathcal{D}_t$ for $t\neq 0$ (Lemma \ref{lemma5}). So $\mathcal{Z}:=\mathcal{C}\times_{\Delta}\mathcal{D}\rightarrow \Delta$ is a required free smoothing.
\end{proof}
\begin{remark}
 In the unmixed case, $G_{(P,Q)}$ is always cyclic.
\end{remark}

According to Theorem \ref{vO} and the discussion thereafter, a surface $X$ is a stable degeneration of surfaces isogenous to a product of unmixed type if and only if $X=(C\times D)/G$, where $C,D$ are stable curves and $G$ is a finite group acting  on $C,D$ and acting diagonally on $C\times D$ such that $(C\times D,G)$ admits a free smoothing. So we have
\begin{corollary}\label{singularityu}
 The possible singularities of a stable degeneration $X$ of surfaces isogenous to a product of unmixed type are as follows:
\begin{enumerate}
 \item [$(U_{1a})$]{Double normal crossing singularities: $(xy=0) \subset\mathbb{C}^3$}. These are the general singularities of $X$.
 \item [$(U_{1b})$]{Quotients of the above singularity under the group action:
\[
 (x,y,z)\mapsto(\xi x,\xi^{-1}y,\xi^q z),
\]
where $\xi$ is a primitive $n$-th root of unity, $(q,n)=1$. In this case, the index of the singularity is $n$ and the canonical covering is the singularity $(U_{1a})$.}
\item [$(U_{2a})$]{The degenerate cusp: $(xy=0,zw=0)\subset\mathbb{C}^4$}.
\item [$(U_{2b})$]{A $\mathbb{Z}_2$-quotient of the degenerate cusp in $(U_{2a})$ under the group action:
\[
 (x,y,z,w)\mapsto(y,x,-z,-w).
\]
In this case, the index of the singularity is $2$ and the canonical covering is the degenerate cusp $(U_{2a})$.}
\item [$(U_{2c})$]{Other quotients of the degenerate cusp $(U_{2a})$ under the group action:
\[
 (x,y,z,w)\mapsto(\xi x,\xi^{-1} y,\xi^q z,\xi^{-q}w),
\]
where $\xi$ is a primitive $n$-th   root of unity, $(q,n)=1$. In this case, the singularity is still a (Gorenstein) degenerate cusp.}
\end{enumerate}
\end{corollary}

We give some examples of the singularities in Cor. \ref{singularityu}.

\begin{example}
 Let $G=\langle\sigma\rangle\cong\mathbb Z_2$. Let $C,D'$ be two hyperelliptic curves. Suppose $\sigma$ acts on $C$ and $D'$ as the respective hyperelliptic involutions. Let  $\{Q'_1,\dots,Q'_{2k}\}$ be the fixed points of $\sigma$ on $D'$.   We obtain a stable curve $D$ from $D'$ by identifying $Q'_{2i-1}$ and $Q'_{2i}$ for any $1\leq i\leq k$. Note that $\sigma$ also acts on $D$. Let $G$ act on $C\times D$ diagonally. Then the quotient $(C\times D)/G$ has singularities of type $(U_{1a})$ or $(U_{1b})$.
\end{example}

\begin{example}
 Let $C,D$ be two stable curves. Let $G$ be a finite group acting freely on $C\times D$. Then $(C\times D)/G$ has singularities of type $(U_{1a})$ or $(U_{2a})$.
\end{example}

\begin{example}
Let $G=\langle\sigma\rangle\cong\mathbb Z_2$. Let $C'$ and $D'$ be two smooth curves of genus $\geq 1$ such that $G$ acts (faithfully) on both. Assume $\sigma$ fixes $2k$  points $P_1',P_2',\dots,P_{2k}'$ on $C'$. Let $C$ be the stable curve obtained by identifying $P'_{2i-1}$ and $P'_{2i}$ for $1\leq i\leq k$. Assume $\sigma$ acts freely on $D'$. Pick a point $Q'\in D'$. Let $D$ be the stable curve obtained by identifying $Q'$ and $\sigma(Q')$. Note that $G$ acts on $C$ and $D$. Let $G$ act on $C\times D$ diagonally. Then $(C\times D)/G$ only has singularities of type $(U_{1a})$ or $(U_{2b})$.
\end{example}

\begin{example}
Let $G=\langle\sigma\rangle\cong\mathbb Z_2$. Let $C'$ and $D'$ be two smooth curves of genus $\geq 1$ such that $G$ acts (faithfully) on both. Assume $\sigma$ fixes $2k$ points $P'_1,P'_2,\dots,P'_{2k}$ on $C'$. We obtain a stable curve $C$ from $C'$ by identifying $P_{2i-1}'$ and $P_{2i}'$ for $1\leq i\leq k$.  Similarly, we can obtain a stable curve $D$ from $D'$. Note that $\sigma$ also acts on $C$ and $D$. Let $G$ acts on $C\times D$ diagonally. Then the quotient $(C\times D)/G$ has singularities of type $(U_{1a})$ or $(U_{2c})$.
\end{example}

\subsection{Singularities of degenerations of surfaces isogenous to a product of mixed type}\label{sectionsmoothingm}

Now we consider the mixed case (cf. Section 1.2).

\begin{proposition}[Criterion for free smoothings in the mixed case]\label{smoothingm}
 Let $C$ be a stable curve and $G^\circ<\text{Aut}(C)$ a finite group.  Let
\[
 1\rightarrow G^\circ\rightarrow G\rightarrow \mathbb{Z}_2 \rightarrow 1,
\]
be a non-split extension, yielding a class $[\varphi]$ in $\text{Out}(G^\circ)=\text{Aut}(G^\circ)/\text{Int}(G^\circ)$, which is of order $\leq 2.$ Fix a representative $\varphi$ of the above class. Suppose there exists an element $\tau'\in G\setminus G^{\circ}$ such that, setting $\tau=\tau'^2$, we have
\begin{enumerate}
 \item [(I)]{$\varphi(\gamma)=\tau'\gamma\tau'^{-1}$},
 \item [(II)]{$G$ acts on $C\times C$ by the formulae: $\gamma(P,Q)=(\gamma P,(\varphi\gamma) Q)$ for $\gamma$ in $G^\circ$; whereas the lateral class of $G^\circ$ consists of the transformations
\[
 \tau'\gamma(P,Q)=((\varphi\gamma)Q,\tau\gamma P).
\] }
\end{enumerate}
Then $(C\times C,G)$ admits a free smoothing if and only if the following hold:
\begin{enumerate}
 \item [(i)]{The pair $(C\times C,G^\circ)$ satisfies one of the properties $(U_0)$, $(U_1)$, $(U_2)$ for any point on $C\times C$, as described in Prop. \ref{smoothingu}.}
 \item [(ii)]{There are only finitely many points with nontrivial stabilizers on $C\times C$.}
 \item [(iii)]{If $(P,Q)\in C\times C$ is such that $G_{(P,Q)}\nsubseteq G^\circ$, then  $P,Q$ are both nodes. }
\end{enumerate}
\end{proposition}

\begin{proof}
 \textquotedblleft $\Rightarrow$ \textquotedblright Let $\mathcal{Z}\rightarrow\Delta$ be a free smoothing of $(C\times C, G)$. It is necessarily of the form $\mathcal C_1\times_{\Delta}\mathcal C_2\rightarrow \Delta$, where $\mathcal C_1\rightarrow \Delta$ is a smoothing of $(C,G^\circ)$ and $\mathcal C_2\rightarrow \Delta$ is a smoothing of $C$ yet with a different $G^\circ$-action given by $G^\circ\xrightarrow{\varphi}G^\circ<\text{Aut}(C)$. Note also that $\mathcal Z\rightarrow\Delta$ is automatically a free smoothing of $(C\times C,G^\circ)$. Hence $(C\times C,G^\circ)$ satisfies $(U_0)$, $(U_1)$, $(U_2)$ in Prop. \ref{smoothingu}.

Let $\Gamma$ be the subset of $G^\circ$ consisting of the transformations having fixed points on $\mathcal{C}_{1t}$ for $t\neq 0 $. Since $\mathcal C_1\times_{\Delta}\mathcal C_2\rightarrow\Delta$ is a free smoothing of $C\times C$, we have
\begin{enumerate}
 \item [(A)]{$\Gamma\cap\varphi(\Gamma)=\{1\}$}.
 \item [(B)]{there is no $\gamma$ in $G^\circ$ such that $\varphi(\gamma)\tau\gamma$ is in $\Gamma$. In particular, $\varphi(\gamma)\tau\gamma\neq 1$.}
\end{enumerate}
The above two conditions simply say that $G$ acts freely on $\mathcal C_{1t}\times\mathcal C_{2t}$ for $t\neq0$.

Suppose there are infinitely many fixed points on $C\times C$. Then some $1\neq\sigma\in G$ fixes infinitely many points and we have $\sigma\in G\setminus G^\circ$ by  Lemma \ref{lemma1}. Hence $\sigma=\tau'\gamma\in G\setminus G^\circ$ for some $\gamma\in G^\circ$. Since $\sigma^2\in G^\circ$ also fixes infinitely many points, we have $\varphi(\gamma)\tau\gamma=\sigma^2=1$ by Lemma \ref{lemma1} again, which contradicts (B) above. This proves (ii).

For (iii), we discuss the possible stabilizer of a point $(P,Q)\in C\times C$ in the following 2 cases.

\textbf{Case \text{$(M_0)$}}: $P,Q$ are both smooth points of $C$.

We will show $G_{(P,Q)}=\{1\}$ in this case. Since $G^\circ$ acts freely on $\mathcal C_{1t}\times\mathcal C_{2t}$ for $t\neq 0$. Note that $G^\circ\cap G_{(P,Q)}=\{1\}$ by the claim for the unmixed $(U_0)$ case. Suppose on the contrary that there is a $1\neq\tau_1\in G_{(P,Q)}$, then $\tau_1=\tau'\gamma\in G\setminus G^\circ$ for some $\gamma\in G^\circ$. Now $\tau_1^2\in G^\circ\cap G_{(P,Q)}$ implies that $\tau_1^2=1$ and hence
\[
 \tau'\gamma\tau'\gamma=1\Rightarrow (\tau'\gamma\tau'^{-1})\tau'^2\gamma=1\Rightarrow \varphi(\gamma)\tau\gamma =1.
\]
This contradicts (B) above.

\textbf{Case \text{$(M_1)$}}: One of $P,Q$ is a node, while the other is a smooth point.

We will show that $G_{(P,Q)}\subset G^\circ$ in this case. Otherwise $(P,Q)$ is fixed by $\tau'\gamma\in G\setminus G^\circ$ for some $\gamma\in G^\circ$:
\[
 (P,Q)=\tau'\gamma(P,Q)=((\varphi\gamma)Q,\tau\gamma P),
\]
so $P=(\varphi\gamma)Q$ and $Q=\tau\gamma P$. In particular, either $P, Q$ are both nodes or they are both smooth points of $C$, a contradiction. Hence (iii) follows.

 \textquotedblleft $\Leftarrow$ \textquotedblright By Prop. \ref{smoothingu}, condition (i) implies that $(C,G^\circ)$ has a smoothing $\mathcal C_1\rightarrow\Delta$ and there is another smoothing $\mathcal C_2\rightarrow \Delta$ of $C$ with a different $G^\circ$-action induced by $\varphi$. These two smoothings of $C$ are in fact isomorphic via $\varphi$. Set $\mathcal{Z}:=\mathcal C_1\times_\Delta\mathcal C_2\rightarrow\Delta$. We can introduce an action of $G$ on $\mathcal{Z}\rightarrow \Delta$ by the formulae in $(II)$: for any $(P,Q)\in \mathcal C_{1t}\times\mathcal C_{2t},\gamma(P,Q)=(\gamma P,(\varphi \gamma)Q)$ if $\gamma\in G^\circ$; whereas for $\tau'\gamma\in G\setminus G^\circ$,
\[
 \tau'\gamma (P,Q)=((\varphi\gamma)Q,\tau\gamma P).
\]

It remains to check that $G$ acts freely on $\mathcal{Z}_t$ for $t\neq 0$. Note that $G^\circ$ acts freely by hypothesis (i) and Prop. \ref{smoothingu}. Now let $\tau'\gamma\in G\setminus G^\circ$ for some $\gamma\in G^\circ$. If $\tau'\gamma$ does not fix points on $C\times C$, then obviously $\tau'\gamma$ does not fix points on $\mathcal{C}_t\times\mathcal{C}_t$ for $t\neq0$. If $\tau'\gamma\in G_{(P,Q)}$ for some $(P,Q)\in C\times C$, then both $P,Q$ must be nodes by (iii) and we can find local embeddings of the first factor: $xy=t^n$ (resp. of the second factor: $zw=t^m$) such that the action of $\tau'\gamma$ around $(P,Q)$ is given by:
\[
 (x,y,z,w,t)\mapsto (az,bw,x,y,t),
\]
where $a,b\in\mathbb{C}^*$ are nonzero numbers. Hence $(\tau'\gamma)^2(x,y,z,w,t)=(ax,by,az,bw,t)$.
If $t\neq 0$, then $xy=t^n,zw=t^m$ implies that $xyzw\neq 0$. Suppose $\tau'\gamma$ fixes some point $(x,y,z,w,t)\in\mathcal C_{1t}\times\mathcal C_{2t}$ for $t\neq 0$, then
\[
 az=x,bw=y,x=z,y=w.
\]
This implies that $a=1,b=1$ and $(\tau'\gamma)^2=1$.  Now, for any $(P',Q')\in \mathcal{Z}_t$, we have
\[
 (\tau'\gamma)^2(P',Q')=((\varphi\gamma)\tau\gamma P', \tau\gamma(\varphi\gamma)Q'),
\]
 so $\tau\gamma(\varphi\gamma)=(\varphi\gamma)\tau\gamma=1$, which implies that $\{((\varphi\gamma)Q',Q')|Q'\in C\}\subset C\times C$ is fixed by $\tau'\gamma$. This contradicts hypothesis (ii) that there are only finitely many points with nontrivial stabilizers. Therefore $\tau'\gamma$ does not fix any points on $\mathcal C_{1t}\times\mathcal C_{2t}$ for $t\neq 0$ and $G$ acts freely on $\mathcal C_{1t}\times\mathcal C_{2t}$ for $t\neq 0$.
\end{proof}

\begin{remark}\label{mixorder}
 Let the notation be as in Prop. \ref{smoothingm}. Then the statement (ii) in the proposition is equivalent to  the assertion that  any $\tau'\gamma\in G\setminus G^\circ$ has order $>2$. Indeed, if there are infinitely many points with nontrivial stabilizers, then some $1\neq \sigma\in G$ fixes infinitely many points. Note that $\sigma\in G\setminus G^\circ$ by Lemma \ref{lemma1}. Since $\sigma^2\in G^\circ$ also fixes infinitely many points, we have $\sigma^2=1$ by Lemma \ref{lemma1} again. On the other hand, suppose $\sigma=\tau'\gamma \in G\setminus G^\circ$ is of order 2. Then $(\varphi\gamma)\tau \gamma = \tau\gamma(\varphi\gamma) =1$ and $\sigma$ fixes every point on the curve $\{((\varphi\gamma) Q,Q)|\ Q\in C\}$.
\end{remark}

According to Theorem \ref{vO}, a surface $X$ is a stable degeneration of surfaces isogenous to a product of mixed type if and only if $X=(C\times C)/G$ where $C$ is a stable curve and $G$ is a finite group acting in the way described in Prop. \ref{smoothingm}.
\begin{corollary}\label{singularitym}
 The possible singularities of a stable degeneration $X$ of surfaces isogenous to a product of mixed type are as follows:
\begin{enumerate}
\item [$(U)$]{The singularities of types $(U_{1a})$, $(U_{1b})$, $(U_{2a})$, $(U_{2b})$, $(U_{2c})$ occurring in the unmixed case (see Cor. \ref{singularityu}). }
\item [$(M)$]{A quotient of the degenerate cusp of type $(U_{2a})$ under an action of automorphisms $\tau_1$ and $\tau_2$:
\begin{align*}
 \tau_1\colon&(x,y,z,w)\mapsto(\xi x,\xi^{-1} y,\xi^q z,\xi^{-q} w),\\
 \sigma\colon&(x,y,z,w)\mapsto(a z, a^{-1} w,  b x, b^{-1}y),
\end{align*}
where $\xi$ is a  primitive  $n$-th root of unity, $(q,n)=1$ and $ab\in \langle\xi\rangle\setminus \langle \xi^{q+1} \rangle$. In this case, the index of the singularity is $2$ and the canonical covering is a singularity of type $(U_{2c})$.}
\end{enumerate}
\end{corollary}
\begin{proof}
Let $(C\times C,G)$ be as in Prop. \ref{smoothingm} such that $X=(C\times C)/G$ and let $\pi:C\times C\rightarrow X$ be the quotient map. If $(P,Q)\in C\times C$ is such that $G_{(P,Q)}\subset G^\circ$, then the singularity $\pi(P,Q)\in X$ is of type $(U_{1a})$, $(U_{1b})$, $(U_{2a})$, $(U_{2b})$ or $(U_{2c})$.

If $(P,Q)\in C \times C$ is such that $G_{(P,Q)}\nsubseteq G^\circ$, then $P,Q$ are both nodes of $C$ by Prop. \ref{smoothingu}. We want to know the action of $G_{(P,Q)}$ around $(P,Q)$.
Note that $(C\times C,G^\circ)$ is of unmixed type and $G^\circ\cap G_{(P,Q)}$ is just the stabilizer of $(P,Q)\in C\times C$ under the action of $G^\circ$. By the analysis done for the unmixed type, $G^\circ\cap G_{(P,Q)}=\langle\tau_1\rangle$ for some $\tau_1\in G^\circ$ and $\tau_1$ interchanges at most the branches of one factors of $C\times C$. We will show that $\tau_1$ does not interchange any branches at $P$ or $Q$.

By assumption there is a $\gamma\in G^\circ$ such that $\sigma:=\tau'\gamma\in G_{(P,Q)}$. If $\tau_1$ interchanges the branches at one of $P,Q$, say $P$, then $|\tau_1|=2$ (Prop. \ref{smoothingu}). Note that
\[
 (\varphi\gamma)\tau\gamma=(\tau'\gamma)^2\in G_{(P,Q)}\cap G^\circ=\langle\tau_1\rangle.
\]
By condition (B) in the proof of Prop. \ref{smoothingm}, $(\varphi\gamma)\tau\gamma\neq 1$. On the other hand, $|\tau_1|=2$, so $(\varphi\gamma)\tau\gamma=\tau_1$. Since $\tau'\gamma\in G_{(P,Q)}$, we have $\tau'\gamma(P,Q)=(P,Q)$, i.e., $(\varphi\gamma)Q=P$ and $(\tau\gamma)P=Q$. Now the fact that $\tau_1$ interchanges the branches of $C$ at $P$ implies that $\tau\gamma(\varphi\gamma)=\tau\gamma\tau_1(\tau\gamma)^{-1}$ interchanges the branches of $C$ at $Q$. Since $\tau_1$ acts on the second factor of $C\times C$ via $\tau\gamma(\varphi\gamma)$, $\tau_1$ also interchanges the branches of the second factor $C$ at $Q$, a contradiction.

So the actions of $\tau_1,\sigma$ are of the form
\begin{align*}
 \tau_1\colon &(x,y,z,w)\mapsto(\xi x,\xi^{-1} y,\xi^q z,\xi^{-q} w),\\
 \sigma\colon &(x,y,z,w)\mapsto(a z, a^{-1} w,b x, b^{-1}y),
\end{align*}
where $C:(xy=0)\subset \mathbb{C}^2$ and $C:(zw=0)\subset \mathbb{C}^2$ are suitable local embeddings of $C$ around $P,Q$ and $\xi $ is a  primitive  $|\tau_1|$-th root of unity, $(q,|\tau_1|)=1$. Since $\sigma^2\in\langle\tau_1\rangle$ and $(\tau^k_1\sigma)^2\neq 1$ for any $k$ (Remark \ref{mixorder}), we can easily see that $ab\in \langle\xi\rangle\setminus \langle \xi^{q+1} \rangle$. Hence the singularity $\pi(P,Q)\in X$ is of type $(M)$.
\end{proof}

We give an example of singularity of type $(M)$.
\begin{example}
 Let $G=\langle\sigma\rangle\cong\mathbb Z_4$. Then $\tau_1:=\sigma^2$ has order 2. Let $C'$ be a smooth curve of genus $\geq 2$. Suppose $\tau_1$ acts on $C'$ so that there are exactly two fixed points $P'_1$ and $P'_2$. We obtain a stable curve $C$ from $C'$ by identifying $P'_1$ and $P'_2$. Let $P\in C$ denote the image  of $P'_1$ and $P'_2$. Then $\tau_1$ also acts on $C$ and has exactly one fixed point $P$.

We can give an action of $G$ on $C\times C$ as follows:
\[
 \sigma(P_1,P_2):=(P_2, \tau_1 P_1),\text{ for any point } (P_1,P_2)\in C\times C.
\] Then $\tau_1(P_1,P_2)=(\tau_1 P_1,\tau_1 P_2)$ and $\sigma\tau_1(P_1,P_2)=(\tau_1 P_2, P_1)$. It is easy to see that $(P,P)\in C\times C$ is the only point with a  nontrivial stabilizer $G$ and the quotient $(C\times C)/G$ has singularities of type $(U_{1a})$ or $(M)$.
\end{example}

\section{Connected components of the moduli space}\label{sectionconnectedcomponent}
In this section we will study the $\mathbb{Q}$-Gorenstein deformations of the stable degenerations of surfaces isogenous to a product. As a result, we get some connected components of the moduli space of stable surfaces $M_{a,b}^{st}$ defined in Section \ref{sectionmoduli}.

Let $Z=C\times D$ be a product of two stable curves and let $G$ be a finite group acting on $Z$ with finitely many fixed points. Let $\pi\colon Z\rightarrow X$ be the quotient map. For any $G$-equivariant coherent sheaf $\mathcal{F}$ on $Z$, we define an $\mathcal{O}_X$-module $\pi^G_*\mathcal{F}:=(\pi_*\mathcal{F})^G$.

The following lemma is well-known.
\begin{lemma}\label{lemmaequivariantcohomology}
 Let $\mathcal{F}$ be a $G$-equivariant coherent sheaf on $Z$. Then for any $p\geq 0$, we have $H^p(Z,\mathcal{F})^G=H^p(X,\pi_*^G\mathcal{F})$.
\end{lemma}
\begin{proof}
 See \cite[Prop. 5.2.2]{G57}.
\end{proof}

\begin{lemma}\label{lemmaequivariantdeformation}
 Suppose all the (possible) singularities on $X$ are of type $(U_{1a})$, $(U_{1b})$, $(U_{2a})$ or $(U_{2b})$. Then $\pi_*^G \mathcal T_Z=\mathcal T_X$ and $\pi_*^G\mathcal{T}^1_Z=\mathcal{T}^1_{QG,X}$.
\end{lemma}
\begin{proof}

First observe that both $\pi_*^G \mathcal T_Z$ and $\mathcal T_X$ are $S_2$-sheaves of $\mathcal{O}_X$-modules (\cite[Lemma 5.1.1]{AbH09}). Since $\pi\colon Z\rightarrow X$ is \'etale off a finite subset, $\pi_*^G \mathcal T_Z$ and $\mathcal T_X$ coincide off the finite subset. Then the $S_2$-property guarantees that $\pi_*^G \mathcal T_Z$ and $\mathcal T_X$ are isomorphic on the whole of $X$.

For $\pi_*^G\mathcal{T}^1_Z=\mathcal{T}^1_{QG,X}$, we view $\pi_*^G\mathcal{T}^1_Z$ (resp. $\mathcal{T}^1_{QG,X}$) as the sheaf of first-order, $G$-equivariant, local deformations of $Z$ (resp. first-order, $\mathbb Q$-Gorenstein, local deformations of $X$). Let $P$ be any point on $X$ and let $\pi^{-1}(P)=\{Q_j\}_j$ be inverse image of $P$. Since the possible singularities on $X$ are of type $(U_{1a})$, $(U_{1b})$, $(U_{2a})$ or $(U_{2b})$, every germ $Q_j\in Z$ is a canonical covering of $P\in X$ and they are permuted under the action of $G$. Since $\mathbb Q$-Gorenstein deformations of the germ $P\in X$ are precisely those deformations which lift to deformations of the canonical covering, we have a natural identification $\pi_*^G\mathcal{T}^1_Z=\mathcal{T}^1_{QG,X}$ sending a first-order, $G$-equivariant, local deformations of $Z$ to its quotient under $G$.
\end{proof}

\begin{theorem}\label{surjective}
If all the (possible) singularities on $X$ are of type $(U_{1a})$, $(U_{1b})$, $(U_{2a})$ or $(U_{2b})$, then a \emph{semiuniversal} $\mathbb Q$-Gorenstein  deformation of $X$ exists and hence the base $\defo_X^{QG}$ is defined. Moreover, $\defo_X^{QG}$ is smooth.
\end{theorem}
\begin{proof}
 Since $Z=C\times D$ is Gorenstein, $\defo_Z^{QG}$ exists and is just $\defo_Z$. Let $f\colon\mathcal Z\rightarrow \defo_Z$ be a semiuniversal deformation of $Z$. Then the action of $G$ on Z induces actions of $G$ on $\mathcal Z$ and $\defo_Z$ such that $f$ becomes a $G$-equivariant morphism. Taking the $G$-invariant part $\defo^G_Z$ of $\defo_Z$ and the $G$-quotient of $f^{-1}((\defo_Z)^G)$, we get a deformation of $X=Z/G$:
\begin{equation}\label{ginvariant}
 f^G\colon f^{-1}((\defo_Z)^G)/G\rightarrow (\defo_Z)^G.
\end{equation}

Note that $\defo_Z=\defo_C\times\defo_D$ is smooth (\cite[Cor. 2.3]{vO05} and \cite{DM}), so $(\defo_Z)^G$ is also smooth by Cartan's lemma. To prove the theorem, it suffices to show that (\ref{ginvariant}) is a \emph{semiuniversal} $\mathbb Q$-Gorenstein  deformation of $X$.

Since all the possible singularities on $X$ are of type $(U_{1a})$, $(U_{1b})$, $(U_{2a})$ or $(U_{2b})$, for any $P\in X$ and $Q\in\pi^{-1}(P)$, the germ $Q\in Z$ is the canonical covering of $P\in X$. So (\ref{ginvariant}) is in fact a $\mathbb Q$-Gorenstein deformation of $X$.

By the infinitesimal lifting property of a smooth variety (\cite[Ex. II.8.6]{Har}), if we can show that the natural map
\[
 d\lambda\colon(T^1_Z)^G\rightarrow T^1_{QG,X}
\]
is an isomorphism, then $f^G\colon f^{-1}((\defo_Z)^G)/G\rightarrow (\defo_Z)^G $ is an unobstructed \emph{semiuniversal} $\mathbb Q$-Gorenstein  deformation of $X$. Consider the following commutative diagram
\[
 \begin{matrix}
  0 &\rightarrow& H^1(Z,\mathcal T_Z)^G &\rightarrow&  (T^1_Z)^G      &\rightarrow&  H^0(Z,\mathcal{T}^1_Z)^G  &\rightarrow& H^2(Z,\mathcal T_Z)^G\\
   &            &      \downarrow\alpha        &         &  \downarrow d\lambda      &           &        \downarrow \beta        &         &    \downarrow \gamma  \\
 0 &\rightarrow&  H^1(X,\mathcal T_X)   &\rightarrow& T^1_{QG,X}      &\rightarrow&   H^0(X,\mathcal{T}^1_{QG,X})&\rightarrow& H^2(X,\mathcal{T}_X)
 \end{matrix}
\]
in which the rows are exact. We will prove that $\alpha,\beta,\gamma$ are isomorphisms, so that $d\lambda\colon(T^1_Z)^G\rightarrow T^1_{QG,X}$ is also an isomorphism by the Five Lemma.

Let $\mathcal{F}=\mathcal T_Z \text{ or } \mathcal{T}^1_Z$ in Lemma \ref{lemmaequivariantcohomology}, we get the following equations
\begin{align*}
 H^1(Z,\mathcal T_Z)^G & =H^1(X,\pi_*^G \mathcal T_Z),\\
 H^0(Z,\mathcal T^1_Z)^G  &=H^0(X,\pi_*^G\mathcal{T}^1_Z),\\
 H^2(Z,\mathcal T_Z)^G   &=H^2(X,\pi_*^G \mathcal T_Z).
\end{align*}

By Lemma \ref{lemmaequivariantdeformation}, we have $\pi_*^G \mathcal T_Z= \mathcal T_X,\pi_*^G\mathcal{T}^1_Z=\mathcal{T}^1_{QG,X}$. Therefore
\begin{align*}
 H^1(Z,\mathcal T_Z)^G  &=H^1(X,\mathcal T_X),\\
 H^0(Z,\mathcal T^1_Z)^G  &=H^0(X,\mathcal T^1_{QG,X}),\\
 H^2(Z,\mathcal T_Z)^G   &=H^2(X,\mathcal T_X),
\end{align*}
and $\alpha,\beta,\gamma$ are isomorphisms.
\end{proof}

\begin{remark}
It remains to address the case where $X$ has singularities of type $(U_{2c})$ or  $(M)$. The canonical coverings of these two types of singularities are not local complete intersections, which results in a more difficult $\mathbb{Q}$-Gorenstein deformation theory. In contrast to the infinitesimal consideration, there might be some hope for good properties of a one-parameter family of such singularities.
\end{remark}

\begin{corollary}\label{connectedcomponent}
 Let $S=(C\times D)/G$ be a surface isogenous to a product of unmixed type. Assume the pair $(C,G)$ is a triangle curve (i.e., $C/G\cong\mathbb{P}^1$, and $C\rightarrow C/G$ is branched over $3$ points). Let $M_S^{\text{top}}$ be the moduli space of smooth surfaces with the same topological type as $S$ and $\overline{M_S^{\text{top}}}$ the closure of $M_S^{\text{top}}$ in the moduli space $M^{st}_{a,b}$ with $a=K_S^2,b=\chi(\mathcal{O}_S)$. Then $\overline{M^{\text{top}}_S}$ consists of connected components of  $M^{st}_{a,b}$.
\end{corollary}
\begin{proof}
 By \cite{Cat03}, every closed point in  $M_S^{\text{top}}$ corresponds to a surface $S'$ isogenous to a product with representation $(C'\times D')/G$. In our case, since a triangle curve is rigid, we can assume $(C',G)$ is also a triangle curve. In fact, $(C',G)=(C,G)\text{ or } (\overline{C},G)$. For the same reason,  $(C',G)$ remains the same in the process of degeneration. If $[X]$ is in $\overline{M_S^{\text{top}}}$, then $X=(C'\times D)/G$, where $D$ is a stable curve. By Prop. \ref{smoothingu} and Cor. \ref{singularityu}, the possible singularities of $X$ are of type $(U_{1a})$ or $(U_{1b})$. Therefore  $\defo^{QG}_X$ is defined and is smooth by Theorem \ref{surjective}. This implies that $M^{st}_{a,b}$ is irreducible at $[X ]$ and the assertion of the corollary follows.
\end{proof}

\end{document}